\documentclass[12pt]{article}
\usepackage{amsmath,amsthm,amssymb}
\usepackage[hyperfootnotes=false]{hyperref}
\usepackage{color}
\usepackage{cancel}
\usepackage{titlesec}
\titleformat{\paragraph}
{\normalfont\normalsize\bfseries}{\theparagraph}{1em}{}
\titlespacing*{\paragraph}
{0pt}{3.25ex plus 1ex minus .2ex}{1.5ex plus .2ex}

\def\titlerunning#1{\gdef\titrun{#1}}
\makeatletter
\def\author#1{\gdef\autrun{\def\and{\unskip, }#1}\gdef\@author{#1}}
\def\address#1{{\def\and{\\\hspace*{18pt}}\renewcommand{\thefootnote}{}%
\footnote {#1}}%
\markboth{\autrun}{\titrun}}
\makeatother
\def\email#1{\hspace*{4pt}{\em e-mail}: #1}
\def\MSC#1{{\renewcommand{\thefootnote}{}%
\footnote{\emph{Mathematics Subject Classification (2020):} #1}}}
\def\keywords#1{\par\medskip
\noindent\textbf{Keywords:} #1}

\newtheorem{theorem}{Theorem}[section]

\newtheorem{prop}[theorem]{Proposition}
\newtheorem{cor}[theorem]{Corollary}
\newtheorem{lemma}[theorem]{Lemma}
\newtheorem{result}[theorem]{Result}

\theoremstyle{definition}

\newtheorem{remark}[theorem]{Remark}

\numberwithin{equation}{section}

\def\cB{\mathcal B}
\def\cC{\mathcal C}

\def\cG{\mathcal G}

\def\cO{\mathcal O}
\def\cP{\mathcal P}
\def\cQ{\mathcal Q}

\def\cR{\mathcal R}
\def\cS{\mathcal S}

\def\PG{{\rm PG}}

\def\GF{{\rm GF}}

\begin{document}


\baselineskip=16pt

\titlerunning{}

\title{Strongly regular graphs from hyperbolic quadrics and their maximal cliques}

\author{
Antonio Cossidente
\and
Jan De Beule
\and
Giuseppe Marino
\and
Francesco Pavese
\and
Valentino Smaldore
}

\date{}

\maketitle

\address{
A. Cossidente: Dipartimento di Ingegneria, Universit{\`a} degli Studi della Basilicata, Contrada Macchia Romana, 85100, Potenza, Italy; \email{antonio.cossidente@unibas.it}
\and
J. De Beule: Department of Mathematics and Data Science, 
Vrije Universiteit Brussel (VUB),  Pleinlaan 2, B--1050 Brussels, 
Belgium and Department of 
Mathematics: Logic, Analysis and Discrete Mathematics, Ghent University, Krijgslaan 281 (S8), B-9000 Gent, Belgium; \email{Jan.De.Beule@vub.be}
\and
G. Marino: Dipartimento di Matematica e Applicazioni ``Renato Caccioppoli'', Universit{\`a} degli Studi di Napoli ``Federico II'', Complesso Universitario di Monte Sant'Angelo, Cupa Nuova Cintia 21, 80126, Napoli, Italy; \email{giuseppe.marino@unina.it}
\and
F. Pavese: Dipartimento di Meccanica, Matematica e Management, Politecnico di Bari, Via Orabona 4, 70125 Bari, Italy; \email{francesco.pavese@poliba.it}
\and
V. Smaldore: Dipartimento di Tecnica e Gestione dei Sistemi Industriali, Università degli Studi di Padova, Stradella S. Nicola 2, 36100 Vicenza, Italy; \email{valentino.smaldore@unipd.it}
}


\MSC{Primary 51E20; Secondary 05E30.}

\begin{abstract}
Let $Q^+(2n+1,q)$ be a hyperbolic quadric of $\PG(2n+1,q)$. Fix a generator $\Pi$ of the quadric.
Define $\cG_n$ as the graph with as vertex set the points of $Q^+(2n+1,q)\setminus \Pi$ and two vertices adjacent
if they either span a secant to $Q^+(2n+1,q)$ or a line contained in $Q^+(2n+1,q)$ meeting $\Pi$ non-trivially. Then such a construction defines a strongly regular graph, which is the complement of a (non-induced) subgraph of the collinearity graph of $Q^+(2n+1,q)$. In this paper, we directly compute the parameters of $\cG_n$, which is cospectral, when $q=2$, to the tangent graph $NO^+(2n+2,2)$, but it is non-isomorphic for $n\geq3$. We also classify the maximal cliques of $\cG_3$ for $q=2$,
proving as a by-product the non-isomorphism with the graph $NO^+(8,2)$.
\end{abstract}

\section{Introduction}
A \textit{strongly regular graph} with parameters $(v,k,\lambda,\mu)$ is a graph with $v$ vertices where each vertex has degree $k$, any two adjacent vertices have $\lambda$ common neighbours, and any two non-adjacent vertices have $\mu$ common neighbours. The \emph{spectrum} of a strongly regular graph is the triple $(k,\theta_{1}^{m_{1}},\theta_{2}^{m_{2}})$. Therefore, two strongly regular graphs with the same parameters are \textit{cospectral}, that is, they have the same spectrum. Two isomorphic graphs are always cospectral, but the converse is not always true. Even in the family of strongly regular graphs, there are examples with the same parameters but not isomorphic. Let $Q^+(2n+1,2)$ be a non–degenerate hyperbolic quadric of $\PG(2n+1,2)$. Let $NO^+(2n +2, 2)$ be the graph whose vertices are the points of $\PG(2n+1,2) \setminus Q^+(2n+1,2)$ and two vertices $P_1$ and $P_2$ are adjacent if the line joining $P_1$ and $P_2$ is a line tangent to $Q^+(2n+1, 2)$. The graph $NO^+(2n +2, 2)$ is a strongly regular graph. In this paper, we provide a construction for strongly regular graphs arising from hyperbolic quadrics. Fix a generator $\Pi$ of $Q^+(2n+1,2)$, and let $\cG_n$ be the graph whose vertices are the points of $Q^+(2n+1,2)\setminus \Pi$, and two vertices $P_1$ and $P_2$ are adjacent if the line joining $P_1$ and $P_2$ is either a secant of $Q^+(2n+1, 2)$, or intersects $\Pi$. It turns out that $\cG_n$ is a strongly regular graph, cospectral to $NO^+(2n+2,2)$. Note that this construction also works in the case $q>2$, giving rise to strongly regular graphs for all quadrics $Q^+(2n+1,q)$. For all values of $q$, the graph was already known as the complement of the graph defined in \cite{BIK}. The main goal of this paper is to give a full geometric characterisation of such a graph. The first interesting case is $(n,q)=(3,2)$, where the graph is cospectral to $NO^+(8,2)$. The so-called {\em ovoids} of $Q^+(7,2)$ are obvious examples of cliques of $\cG_3(2)$ (which also turn out to be cliques of maximal size), and this observation motivates us to classify and geometrically characterise all maximal cliques of $\cG_n$ for $(n,q)=(3,2)$. This characterisation includes a full description of all types of maximal cliques in terms of substructures of the underlying quadric. As a by-product, this complete classification gives the non-isomorphism between $\cG_n$ and $NO^+(2n+2,2)$ whenever $n \geq 3$.

The paper is outlined as follows. In Section \ref{sec2} we give the necessary preliminaries on spectral graph theory of strongly regular graphs and projective geometry of hyperbolic quadrics, while in Section \ref{sec3} we compute the parameters of $\cG_n$. 
Finally, in Section \ref{sec4} we compute cliques of $\cG_3$, thus proving the non-isomorphism issue. 

\section{Preliminaries}\label{sec2}

\subsection{Spectra of strongly regular graphs}
A graph $\Gamma$ is an ordered pair $(V(\Gamma),E(\Gamma))$, where $V(\Gamma)$ is a non-empty set of vertices and $E(\Gamma)$ is the set of edges, i.e. unordered pairs of vertices. If $e=uv$ we say that the edge $e$ joins $u$ and $v$, and $u$ and $v$ are called adjacent vertices or neighbours.  A \textit{clique} of the graph $\Gamma$ is a set of pairwise adjacent vertices. A clique is said to be \textit{maximal} if it is maximal with respect to set theoretical inclusion. 

Spectral graph theory is a widely used combinatorial approach to the study of graphs based on the eigenvalues of their adjacency matrix. We refer the reader to \cite{BH} for more details on spectral graph theory. Strongly regular graphs were introduced by R. C. Bose in \cite{Bose} in 1963, and ever since then they have been extensively investigated. In particular, the eigenvalues of the adjacency matrix of a strongly regular graph are known; see \cite{brvan}:
     a strongly regular graph $G$ with parameters $(v,k,\lambda,\mu)$ has exactly three eigenvalues: $k$, $\theta_{1}$ and $\theta_{2}$ of multiplicity, respectively, $1$, $m_{1}$ and $m_{2}$, where:
    $$\theta_{1}=\frac{1}{2}\big[(\lambda-\mu)+\sqrt{(\lambda-\mu)^{2}+4(k-\mu)}\big],$$
    $$\theta_{2}=\frac{1}{2}\big[(\lambda-\mu)-\sqrt{(\lambda-\mu)^{2}+4(k-\mu)}\big],$$
    $$m_{1}=\frac{1}{2}\Big[(v-1)-\frac{2k-(v-1)(\lambda-\mu)}{\sqrt{(\lambda-\mu)^{2}+4(k-\mu)}}\Big],$$
    $$m_{2}=\frac{1}{2}\Big[(v-1)+\frac{2k-(v-1)(\lambda-\mu)}{\sqrt{(\lambda-\mu)^{2}+4(k-\mu)}}\Big].$$
 In an analogous way, we can express the parameters as a function of the spectrum. Since spectrum and parameters of strongly regular graphs are equivalent, two strongly regular graphs are cospectral if and only if they have the same parameters. It is still an open problem the characterisation of strongly regular graphs \emph{determined by their spectrum}. 

\subsection{Projective spaces and hyperbolic quadrics}

Let $\PG(r - 1, q)$ be the projective space of projective dimension $r - 1$ over $\GF(q)$ equipped with homogeneous projective coordinates $X_1, \dots, X_{r}$. We will use the term $n$--space of $\PG(r - 1, q)$ to denote an $n$--dimensional projective subspace of $\PG(r - 1, q)$. Let $\theta_r$ be the number of points of $\PG(r-1,q)$, then
\[\theta_r=\frac{q^r-1}{q-1}.\]
We shall find it helpful to represent the projectivities of $\PG(r - 1, q)$ by invertible $r \times r$ matrices over $\GF(q)$ and consider the points of $\PG(r - 1, q)$ as column vectors, with matrices acting on the left. Let $U_i$ be the points that have $1$ in the $i$--th position and $0$ elsewhere.

Let $Q^+(2n+1,q)$ be a hyperbolic quadric of $\PG(2n+1,q)$. Up to choice of basis, this quadric is the set of points of the variety given by
\begin{equation*}
    Q: X_1X_{2n+2}+X_2X_{2n+1}+\ldots+X_nX_{n+1}=0.
\end{equation*}
A projective subspace of maximal dimension contained in $\cQ$ is called a {\em generator} of $\cQ$, and the projective dimension of the generators of 
$\cQ$ equals $n$. The number of points of $Q^+(2n+1,q)$ equals $\frac{(q^n+1)(q^{n+1}-1)}{q-1}$, while the number of generators equals 
$\prod_{i=1}^{n+1}(q^{n-i+1}+1)=2(q+1)(q^2+1)\cdot...\cdot(q^{n}+1)$. The generators split into two families of size $(q+1)(q^2+1)\cdot...\cdot(q^{n}+1)$. 
An important property of the hyperbolic quadric (and all finite classical polar spaces), is the one-or-all property (or axiom), i.e., given a point $P$ and a line $l$
of the polar space, $P \not \in l$, then either $P$ is collinear with exactly one point of $l$ or with all points of $l$. Hence, if $P$ is collinear with two points of $l$, then the plane $\langle l, P \rangle$ is completely contained in the polar space. 

An ovoid of $Q^+(2n+1,q)$ is a set $\cO$ of points such that every generator of $Q^+(2n+1,q)$ meets $\cO$ in exactly one point. Necessarily, $\cO$ contains $q^n+1$ points. Ovoids in quadrics (and more generally finite classical polar spaces) are rare, but for $n=3$ and $q=2$, there is, up to isomorphism,
a unique ovoid of $Q^+(7,2)$, which will play an important role in classifying cliques of $\cG_3$ for $q=2$. 

For $n=2$, an ovoid of $Q^+(5,q)$ is by Klein correspondence equivalent with a spread of the projective
space $\PG(3,q)$. Hence there are many examples of ovoids of $Q^+(5,q)$. An important example for 
this paper is the elliptic quadric $Q^-(3,q)$, which is the set of points of the variety given by 
\begin{equation*}
    Q: F(X_1,X_2) + X_3X_4=0,
\end{equation*}
with $F(X_1,X_2)$ a homogeneous and irreducible quadratic form over the field $\GF(q)$. 

The quadratic form $X_1X_{2n+2}+X_2X_{2n+1}+\ldots+X_nX_{n+1}$ defining the quadric has associated polar form $f$, which is non-degenerate and
orthogonal in odd characteristic, symplectic in even characteristic. The form $f$ gives rise to a polarity of the ambient space $\PG(2n+1,q)$, denoted as 
$\perp$, and with the property that for any $d$-subspace $\pi$ contained in $\cQ$, the space $\pi^\perp$ is the tangent space to $\cQ$ at $\pi$. 
We will rely on the well-known intersection properties of $\pi^\perp$ and the quadric $\cQ$ for various subspaces $\pi$ contained or not contained
in $\cQ$. For details on hyperbolic quadrics, we refer the reader to \cite{HT}. When $n=3$, the quadric $Q^+(7,q)$ is the set of points of the variety given by:
\begin{equation*}
    Q: X_1X_8+X_2X_7+X_3X_6+X_4X_5=0.
\end{equation*}
The quadric contains $(q+1)(q^2+1)(q^3+1)$ points, and the generators of the quadric are the $2(q+1)(q^2+1)(q^3+1)$ solids, which split into two families of size $(q+1)(q^2+1)(q^3+1)$. An ovoid of size $q^3+1$ is equivalent to a spread of $q^3+1$ solids, see \cite{CP}. If $q=2$, $Q^+(7,2)$ contains 135 points, 270 generator solids, 135 in each family, and the ovoids (and spreads) have size 9.

\section{The strongly regular graph $\cG_n$}\label{sec3}

Consider the quadric $Q^+(2n+1,q)$ with $\perp$ its induced polarity of $\PG(2n+1,q)$. Fix a generator $\Pi$ of $Q^+(2n+1,q)$. The set of vertices $V$ of $\cG_n$ consists of the points of
$Q^+(2n+1,q) \setminus \Pi$. We define two relations in $V$. Let $P,Q \in V$, then $P \sim_1 Q$ if and only if the line $\langle P,Q \rangle$ is secant to $Q^+(2n+1,q)$, and
$P \sim_2 Q$ if and only if the line $\langle P,Q \rangle$ is contained in $Q^+(2n+1,q)$ and meets the generator $\Pi$ in a point. The adjacency relation of $\cG_n$
is the union of $\sim_1$ and $\sim_2$. 

\begin{theorem}\label{mainth}
The graph $\cG_n$ is a strongly regular graph with parameters $v=\frac{(q^n+1)(q^{n+1}-1)}{q-1}-\frac{q^{n+1}-1}{q-1}=\frac{q^n(q^{n+1}-1)}{q-1}$, $k=q^{2n}-1$, $\lambda =q^{2n-1}(q-1)-2$
and $\mu = (q^{2n-1}+q^{n-1})(q-1)$.
\end{theorem}
\begin{proof}
Let us start by determining the valency $k$ of $\cG_n$.

Let $P$ be a point of $Q^+(2n+1,q)\setminus\Pi$. Then $\Pi\cap P^\perp$ is an $(n-1)$-space, say $\overline{\Pi}$. Let $\Pi'$ be a $(2n-1)$-space of $P^\perp$ through $\overline{\Pi}$ and not passing through $P$, then $P^\perp\cap Q^+(2n+1,q)$ is a quadratic cone with as vertex $P$ and as basis a hyperbolic quadric $Q^+(2n-1,q)$ contained in $\Pi'$ and of which $\overline{\Pi}$ is a totally isotropic space. In order to determine the valency $k$ of $\cG_n$, we have:
\begin{itemize}
    \item to count the number of points $Q^+(2n+1,q)$ not contained in $P^\perp\cap Q^+(2n+1,q)$,
    \item to delete the number of points of $\Pi\setminus \overline{\Pi}$,
    \item to add the points of $P^\perp\cap Q^+(2n+1,q)$, different from $P$, lying on lines through $P$ intersecting the space $\overline{\Pi}$ and not contained in $\Pi'$.
\end{itemize}
Hence, we have
\begin{align*}
k&=|Q^+(2n+1,q)|-(q|Q^+(2n-1,q)|+1)-(\theta_{n}-\theta_{n-1})+(q-1)\theta_{n-1}\\
&=|Q^+(2n+1,q)|-\theta_{n}-q|Q^+(2n-1,q)|+q\theta_{n-1}-1\\
&=\frac{(q^n+1)(q^{n+1}-1)}{q-1}-\frac{q^{n+1}-1}{q-1}-q\frac{(q^{n-1}+1)(q^n-1)}{q-1}+q\frac{q^n-1}{q-1}-1=q^{2n}-1.
\end{align*}

Now we want to determine \underline{the value of $\lambda$}.

\fbox{Case I}: Let $P_1$ and $P_2$ be two points of $Q^+(2n+1,q)$ not in $\Pi$ such that the line $l:=P_1P_2$ is secant to $Q^+(2n+1,q)$. We want to determine the number of vertices of $\cG_n$ adjacent to $P_1$ and $P_2$.

Note that $|l\cap \Pi|=0$, $|l^{\perp}\cap Q^+(2n+1,q)|=|Q^+(2n-1,q)|$ and $l^\perp\cap\Pi$ is an $(n-2)$-space, say $\overline{\Pi}$, which is totally isotropic of $Q^+(2n-1,q)$.

We first show that there are no point $X$ of $\cG_n$ adjacent to $P_1$ and $P_2$ such that the lines $XP_1$ and $XP_2$ both intersect the space $\Pi$. Indeed, in such a case the line $P_1P_2$ would intersect $\Pi$, a contradiction.

The parameter $\lambda$ is given by $\lambda:=x-y+z_1+z_2$, where
\begin{align*}
    x\quad & \mbox{is the number of points $X$ of $Q^+(2n+1,q)$ such that the lines $P_1X$ and $P_2X$ are both}\\
    &\mbox{secant lines to $Q^+(2n+1,q)$}\\
    y\quad & \mbox{is the number of points of $\Pi\setminus (P_1^\perp\cup P_2^\perp)$}\\
    z_1\quad & \mbox{is the number of points of $P_1^\perp \cap Q^+(2n+1,q) \setminus (\{P_1\}\cup \Pi\cup P_2^\perp)$}\\
    z_2\quad & \mbox{is the number of points of $P_2^\perp \cap Q^+(2n+1,q) \setminus (\{P_2\}\cup\Pi\cup P_1^\perp)$}.
\end{align*}

The integer $x$ is obtained by subtracting from the number of points in $Q^+(2n+1,q)\setminus\{P_1,P_2\}$ the number of points of $l^\perp\cap Q^+(2n+1,q)$ and the size of $(P_i^\perp\cap Q^+(2n+1,q)\setminus\{P_i\})\setminus P_j^\perp$, where $i,j\in\{1,2\}$ with $i\ne j$. Then
\begin{align}\label{form:x}
x=& |Q^+(2n+1,q)|-2-|Q^+(2n-1,q)|-2(q-1)|Q^+(2n-1,q)|\nonumber\\
 =& \frac{(q^n+1)(q^{n+1}-1)}{q-1}-2+(1-2q)\frac{(q^{n-1}+1)(q^n-1)}{q-1}.
\end{align}

Since $\Pi\cap P_i^\perp$ is an $n$-space and $P_1^\perp\cap P_2^\perp=l^\perp$ which is an $(n-2)$-space, we get
\begin{equation}\label{form:y}
    y=\theta_n-q^{n-1}-q^{q^n-1}-\theta_{n-2}=q^{n-1}(q-1).
\end{equation}

    The integer $z$ is given by $(P_1^\perp\cap Q^+(2n+1,q)) \setminus (\{P_1\}\cup \Pi\cup l^\perp)$. Since $\Pi\cap P_1^\perp$ is an $(n-1)$-space, $l^\perp\subset P_1^\perp$ and $M:=l^\perp\cap\Pi$ is an $(n-2)$-space, any line of $Q^+(2n+1,q)$ through $P_1$ intersecting $M$ gives a contribution of $q-1$ points and any line of $Q^+(2n+1,q)$ through $P_1$ intersecting $\Pi$ also intersects $l^\perp$ and hence gives a contribution of $q-2$ points, we get

\begin{equation}\label{form:z}
    z_1=(q-1)\theta_{n-2}+(q-2)q^{n-1}.
\end{equation}

Arguing as above, we get $z_2=z_1$. Hence from Equations \eqref{form:x}, \eqref{form:y} and \eqref{form:z} we have
\begin{align*}
\lambda=&\frac{(q^n+1)(q^{n+1-1})}{q-1}-2+(1-2q)\frac{(q^{n-1}+1)(q^n-1)}{q-1}-q^{n-1}(q-1)+2(q-1)\frac{q^{n-1}-1}{q-1}\\
&+2(q-2)q^{n-1}=q^{2n-1}(q-1)-2.
\end{align*}

\fbox{Case II}: Let $P_1$ and $P_2$ be two points of $Q^+(2n+1,q)$ not in $\Pi$ such that the line $l:=P_1P_2$ intersects $\Pi$ at a point, say $R$, and therefore is totally isotropic. We want to determine the number of vertices of $\cG_n$ adjacent to $P_1$ and $P_2$.

Note that $|l\cap \Pi|=1$, $l^\perp\cap Q^+(2n+1,q)$ is a quadratic cone with vertex $l$ and basis a hyperbolic quadric $Q^+(2n-3,q)$. Also, $l^\perp\cap \Pi$ is an $(n-1)$-space, say $\overline{\Pi}$.

The parameter $\lambda$ is given by $\lambda:=x-y+z+t$, where
\begin{align*}
    x\quad & \mbox{is the number of points $X$ of $Q^+(2n+1,q)$ such that the lines $P_1X$ and $P_2X$ are both}\\
    &\mbox{secant lines to $Q^+(2n+1,q)$}\\
    y\quad & \mbox{is the number of points of $\Pi\setminus (P_1^\perp\cup P_2^\perp)$}\\
    z\quad & \mbox{is the number of points of $l \setminus (\{P_1,P_2,R\}$}\\
    t\quad & \mbox{is the number of points of $\langle l,\overline{\Pi}\rangle \setminus (\overline{\Pi}\cup l)$},
\end{align*}

The integer $x$ is obtained by subtracting from the number of points in $Q^+(2n+1,q)$ the number of points of $l^\perp\cap Q^+(2n+1,q)$ and the size of $(P_i^\perp\cap Q^+(2n+1,q)\setminus\{P_i\})\setminus P_j^\perp$, where $i,j\in\{1,2\}$ with $i\ne j$. Then
\begin{align}\label{form1:x}
x=& |Q^+(2n+1,q)|-q^2|Q^+(2n-3,q)|-(q+1)-2(q|Q^+(2n-1,q)|+1-(q^2|Q^+(2n-3,q)|+q+1))\nonumber\\
 =& \frac{(q^n+1)(q^{n+1}-1)}{q-1}-(q+1)+q^2\frac{(q^{n-2}+1)(q^{n-1}-1)}{q-1}-2q\frac{(q^{n-1}+1)(q^n-1)}{q-1}+2q\nonumber\\
 =&q^{2n-1}(q-1).
\end{align}

Since $\Pi\cap P_1^\perp=\Pi\cap P_2^\perp=\Pi\cap l^\perp=\overline{\Pi}$ is an $(n-1)$-space  we get
\begin{equation}\label{form1:y}
    y=\theta_n-\theta_{n-1}=q^{n}.
\end{equation}

It is clear that $z=q-2$ and since $l\cap\overline{\Pi}$ is the point $R$ we get
\[t=\theta_n-\theta_{n-1}-q=q^{n}-q.\]

 Hence, from the last two equalities and from Equations \eqref{form1:x} and \eqref{form1:y} we get
\[
\lambda=q^{2n-1}(q-1)-q^n+q-2+q^n-q=q^{2n-1}(q-1)-2.
\]

Finally, we want to determine \underline{the value of $\mu$}.

Let $P_1$ and $P_2$ be two points of $Q^+(2n+1,q)$ such that the line $l:=P_1P_2$ is totally isotropic and it is disjoint from $\Pi$. We want to determine the number of vertices of $\cG_n$ adjacent to $P_1$ and $P_2$.

Note that  $l^\perp\cap Q^+(2n+1,q)$ is a quadratic cone with vertex $l$ and basis a hyperbolic quadric $Q^+(2n-3,q)$. Also, $l^\perp\cap \Pi$ is an $(n-2)$-space, say $\overline{\Pi}$.

The parameter $\mu$ is given by $\mu:=x-y+z_1-z_2-z_3+t_1-t_2-t_3$, where
\begin{align*}
    x\quad & \mbox{is the number of points $X$ of $Q^+(2n+1,q)$ such that the lines $P_1X$ and $P_2X$ are both}\\
    &\mbox{secant lines to $Q^+(2n+1,q)$}\\
    y\quad & \mbox{is the number of points of $\Pi\setminus (P_1^\perp\cup P_2^\perp)$}\\
    z_1\quad & \mbox{is the number of points of the quadratic cone with vertex $P_1$ and basis the $(n-1)$-space $P_1^\perp\cap \Pi$}\\
    z_2\quad & \mbox{is the number of points of $P_1^\perp\cap \Pi$}\\
    z_3\quad & \mbox{is the number of points of the quadratic cone with vertex $P_1$ and basis the $(n-2)$-space $l^\perp\cap \Pi$}\\
    & \mbox{not contained in $\overline{\Pi}$}\\
    t_1\quad & \mbox{is the number of points of the quadratic cone with vertex $P_2$ and basis the $(n-1)$-space $P_2^\perp\cap \Pi$}\\
    t_2\quad & \mbox{is the number of points of $P_2^\perp\cap \Pi$}\\
    t_3\quad & \mbox{is the number of points of the quadratic cone with vertex $P_2$ and basis the $(n-2)$-space $l^\perp\cap \Pi$}\\
    & \mbox{not contained in $\overline{\Pi}$}.
\end{align*}

The integer $x$ is obtained by subtracting from the number of points in $Q^+(2n+1,q)$ the number of points of $l^\perp\cap Q^+(2n+1,q)$ and the size of $(P_i^\perp\cap Q^+(2n+1,q)\setminus\{P_i\})\setminus P_j^\perp$, where $i,j\in\{1,2\}$ with $i\ne j$. Then, as in the previous case, we get
\begin{align}\label{form2:x}
x=& |Q^+(2n+1,q)|-q^2|Q^+(2n-3,q)|-(q+1)-2(q|Q^+(2n-1,q)|+1-(q^2|Q^+(2n-3,q)|+q+1))\nonumber\\
 =& \frac{(q^n+1)(q^{n+1}-1)}{q-1}-(q+1)+q^2\frac{(q^{n-2}+1)(q^{n-1}-1)}{q-1}-2q\frac{(q^{n-1}+1)(q^n-1)}{q-1}+2q\nonumber\\
 =&q^{2n-1}(q-1).
\end{align}

Since $\Pi\cap P_1^\perp$ and $\Pi\cap P_2^\perp$ are two $(n-1)$-spaces and $\overline{\Pi}=\Pi \cap l^\perp$ is an $(n-2)$-space  we get
\begin{equation}\label{form2:y}
    y=\theta_n-2\theta_{n-1}+\theta_{n-2}=q^{n-1}(q-1).
\end{equation}

It can be easily seen that

\begin{align*}
    &z_1=t_1=q\theta_{n-1}+1,\\
    &z_2=t_2=\theta_{n-1},\\
    &z_3=t_3=(q-1)\theta_{n-2}+1,
\end{align*}

and hence $z_1-z_2-z_3=t_1-t_2-t_3=q^{n-1}(q-1)$. From these equalities and from Equations \eqref{form2:x} and \eqref{form2:y}, we get
\[\mu=q^{2n-1}(q-1)-q^{n-1}(q-1)+2q^{n-1}(q-1)=(q^{2n-1}-q^{n-1})(q-1).\]
\end{proof}

\begin{remark}\label{ihr}
    Let $\overline{\cG}_n$ be the complement of $\cG_n$. Then the vertices of $\overline{\cG}_n$ are the points of $Q^+(2n+1, q) \setminus \Pi$ and two vertices are adjacent in $\overline{\cG}_n$ whenever the line joining them is contained in $Q^+(2n+1, q)$ and skips $\Pi$. Therefore $\overline{\cG}_n$ is a (not-induced) subgraph of the collinearity graph of $Q^+(2n+1, q)$. Hence, the strong regularity was already stated in \cite{BIK}. The argument above reported involves the complementary graph and it provides a full proof of \cite[Theorem 1]{BIK} by the following identity: the complement of a $srg(v,k,\lambda,\mu)$ is strongly regular with parameters $(v,v-k-1,v-2k+\mu-2,v-2k+\lambda)$. 
 \end{remark}

\section{The graphs $NO^{+}(2n+2,2)$ and $\cG_n$ for $q=2$}\label{sec4}

Consider again the quadric $Q^{+}(2n+1,2)$ embedded in $\PG(2n+1,2)$. The graph $NO^{+}(2n+2,2)$ has vertex set the set of points of $\PG(2n+1,2)\setminus Q^{+}(2n+1,2)$,
and two vertices $P,Q$ are adjacent if and only if $P \in Q^\perp$, equivalently, if and only if the line $\langle P,Q \rangle$ is a tangent line to $Q^{+}(2n+1,2)$. 
The family of graphs $NO^{+}(2n+2,2)$ (for $n \geq 1$) is described in \cite{brvan}, they are strongly regular with parameters 
$v=2^{2n+1}-2^{n}$, $k=2^{2n}-1$, $\lambda=2^{2n-1}-2$, $\mu=2^{2n-1}+2^{n-1}$, the same parameters as the graphs $\cG_n$ for $q=2$. 
One can ask whether the graphs are isomorphic or not.
\begin{prop}
The graph $\cG_n$ is isomorphic to $NO^{+}(2n+2,2)$ if $n=0,1,2$.
\end{prop}
\begin{proof}
For $n=0$ both graphs consist of one vertex, for $n=1$ both graphs are the complete bipartite graph $K_{3,3}$. The first
non-trivial case is $n=2$, for which it is shown in \cite{RS} that both graphs are isomorphic. 
\end{proof}

Now consider the case $n=3$. We will show that both graphs have maximal cliques of different sizes, which shows that they 
are non-isomorphic. Recall that a {\em clique} is a set of vertices in which every pair of vertices is adjacent. A clique 
is {\em maximal} if it cannot be extended to a clique of larger size. Recall the following well known theorem on the
largest possible size of cliques. 

\begin{result}[\textbf{Delsarte clique bound}, {\cite[Section 3.3.2]{Delsarte}}]\label{DB}
Let $\Gamma$ be a strongly regular graph with regularity $k$ and smallest eigenvalue $\lambda$. 
Then the size of a clique in $\Gamma$ is at most $1-\frac{k}{\lambda}$. 
\end{result}

\begin{cor}\label{cor:maxclique}
The size of a clique in $NO^+(8,2)$ and $\cG_3$ is at most 8.
\end{cor}

\begin{proof}
Use Result \ref{DB} with $k=63$ and $\lambda=-9$.
\end{proof}

Note that a maximal clique is not necessarily a clique of largest possible size. In \cite{brvan} the maximal cliques of $NO^{+}(2n+2,2)$ 
are characterised and appear as follows. Let $\pi$ be an $n$-space of $\PG(2n+2,q)$ meeting $Q^+(2n+1,q)$ in an $(n-1)$-space. 
The $2^n$ points of $\pi \setminus Q^+(2n+1,q)$ are a clique of size $2^n$, necessarily maximal, and no other maximal cliques exist
in $NO^{+}(2n+2,2)$.

\subsection{Classification of maximal cliques of $\cG_3$}

It is a relatively easy exercise to set up the graph $\cG_3$ for $q=2$ using GAP and FinInG 
(\cite{GAP,fining}) and to obtain a complete classification of all maximal cliques of $\cG_3$. 
Furthermore, the results obtained by computer allow us to better understand the geometrical structure 
of the cliques and inspire us to give a computer free proof of the classification, including a geometrical characterisation
in terms of substructures of the quadric. We have created a github archive
containing all necessary code and results; see \cite{jdbgit25}. The computational 
approach gives rise directly to the classification results
in Table~\ref{cliqueTab}, which lists the five orbits that the automorphism group of the 
graph has on the 12777 maximal cliques. 
\begin{table}
\centering
\begin{center}
\begin{tabular}{|c|c|c|c|}
\hline
\textbf{\# of cliques} & \textbf{Size} & \textbf{Adjacencies} & \textbf{Geometric description}   \\ \hline
10752 & 5 & $\sim_1$& $Q^-(3,2)\subseteq Q^+(7,2)$ not meeting $\Pi$\\ \hline
960 & 8 & $\sim_1$ & Cameron-Praeger ovoid \cite{CP} \\ \hline
15 & 8 & $\sim_2$ & Generators of $Q^+(7,2)$ meeting $\Pi$ on a plane \\ \hline
840 & 8 & mixed & Cones $PQ^-(3,2)$, $P\in\Pi$, meeting $\Pi$ in a line \\ \hline
210 & 8 & mixed & Cones $l Q^+(1,2)$, $l\in\Pi$, meeting $\Pi$ in $l$  \\ \hline

\end{tabular}
\caption{\footnotesize{List of maximal cliques of $\cG_3$. In each case we consider all points in the geometric description \emph{out} of $\Pi$.}}\label{cliqueTab}
\end{center}
\end{table}

 Recall that the adjacency of the graph $\cG_n$ is the union of the two relations $\sim_1$ (vertices
defining a secant line) and $\sim_2$ (vertices defining a line contained in the quadric and meeting the fixed generator $\Pi$ in
a point).  Vertices in a clique of $\cG_3$ can hence be in one of the relations $\sim_1$ and $\sim_2$. If all pairs of a clique $\cC$ are in $\sim_1$, 
$\sim_2$ respectively, then we say that $\cC$ is of type $1$, respectively type $2$, otherwise we call $\cC$ of {\em mixed} type.

\begin{lemma}
    All the numbers in Table \ref{cliqueTab} are correct.
\end{lemma}

\begin{proof}
 In $Q^+(7,2)$ we have 24192 elliptic sections $Q^-(3,2)$, 896 through each point. As $|\Pi|=15$, we obtain the 10752 $Q^{-}(3,2)$ not intersecting $\Pi$ as $24192-(15\cdot896)$. The number of ovoids is proven in \cite{CP}, while the 15 listed configurations of type 2 are in one-to-one correspondence to the 15 planes in $\Pi$. We now want to count mixed type configurations. Fix a line $\overline{l}\subseteq\Pi$. Then $P^\perp$ is a 6-space that meets the quadric $Q^+(7,2)$ in a cone $PQ^+(5,2)$, which intersects $\Pi$ in $\overline{l}$. The subquadrics $Q^-(3,2)\subseteq Q^+(5,2)$ correspond to the 56 lines in $\PG(5,2)$ skew to $Q^+(5,2)$. The 56 quadrics $Q^-(3,2)$ are exactly 8 on each point, and we choose the 8 quadrics through the point $\overline{l}\cap\PG(5,2)$. Then, since $\Pi$ contains 35 lines and each line contains 3 points, we get the $35\cdot3\cdot8=840$ such configurations $PQ^-(3,2)$. Now fix a line $l$ in $\Pi$, then $l^{\perp}$ is a 5-space that meets the quadric $Q^+(7,2)$ in a cone $l Q^+(3,2)$. The basis $Q^+(3,2)$ meets $\Pi$ in a line $r$. The bases $Q^+(1,2)$ needed correspond to the 6 secant lines in $Q^+(3,2)$ that do not intersect the fixed $r$. Hence, we have $210=35\cdot6$ configurations $l Q^+(1,2)$, as $\Pi$ contains 35 lines.  
\end{proof}
In the remainder of this section we will prove in a geometric way that Table \ref{cliqueTab} provides the whole classification of the cliques of $\cG_3$.

\begin{lemma}\label{le:ovoid1}
An ovoid of $Q^+(7,2)$ gives rise to a maximal clique of size $8$ of $\cG_3$.
\end{lemma}
\begin{proof}
Let $\cO$ be an ovoid of $Q^+(7,2)$. Then every pair of points $P,Q \in \cO$ determines necessarily a secant to $Q^+(7,2)$. 
The ovoid $\cO$ meets the generator $S$ in a unique point, hence $\cO$ contains exactly $8$ vertices that are mutually in relation $\sim_1$. 
This is a clique of size $8$, the largest possible size by Corollary~\ref{cor:maxclique}.
\end{proof}

A {\em partial ovoid} of $Q^+(2n+1,q)$ is a set $\cB$ of points such that no two points of $\cB$ are collinear on the quadric. A partial ovoid
is {\em maximal} if it cannot be extended to a larger partial ovoid. Obviously, an ovoid is a maximal partial ovoid, any subset of an ovoid is a 
partial ovoid which is not maximal. 

\begin{lemma}\label{le:ovoid2}
A partial ovoid of $Q^+(7,2)$ gives rise to a clique in the relation $\sim_1$. 
\end{lemma}
\begin{proof}
Similar to the proof of Lemma~\ref{le:ovoid1}.
\end{proof}

It is not a priori clear whether cliques from maximal partial ovoids are maximal or can be extended by 
adding vertices that are in relation $\sim_2$ with the vertices of the initial partial ovoid. However, we can easily construct cliques of size $5$ in $\sim_1$.

\begin{lemma}\label{le:elliptic1}
An elliptic quadric $Q^-(3,q)$ contained in $Q^+(7,q)$ is a maximal partial ovoid of $Q^+(7,q)$. 
\end{lemma}
\begin{proof}
Assume that some solid $\pi_3$ of $\PG(7,q)$ meets $Q^+(7,q)$ in an elliptic quadric $Q^-(3,q)$. The points 
of $Q^-(3,q)$ are a partial ovoid of $Q^+(7,q)$. Now, let $P \in Q^+(7,q) \setminus \pi_3$. Then $P^\perp$
meets $\pi_3$ in a plane, which meets $Q^-(3,q)$ either in a point or in a conic. Hence, at least one point
of $Q^-(3,q)$ is collinear in $Q^+(7,q)$ with $P$, so no point of $Q^+(7,q) \setminus Q^-(3,q)$ can be 
added to extend $Q^-(3,q)$ as a partial ovoid.
\end{proof}

\begin{lemma}\label{le:max5}
Let $q=2$ and let $Q^-(3,q)$ be an elliptic quadric contained in $Q^+(7,q)$, 
not meeting $\Pi$. Then the points of $Q^-(3,q)$ are a maximal clique of size $5$.
\end{lemma}
\begin{proof}
By Lemma~\ref{le:elliptic1} the points of $Q^-(3,q)$ are a maximal partial
ovoid of size $5$ contained in the vertex set of $\cG_3$, and hence by Lemma~\ref{le:ovoid2}
this gives a clique of size $5$ in relation $\sim_1$. Assume that a vertex $T$ extends 
$Q^-(3,q)$ to a larger clique. It is not possible that $T$ 
extends this clique to a larger clique in $\sim_1$, since this would extend $Q^-(3,q)$ to a larger
partial ovoid, a contradiction by the maximality of $Q^-(3,q)$ as partial ovoid. Hence 
$T \sim_2 P$ for some point $P \in Q^-(3,q)$. But then the line $\langle P,T \rangle$ meets 
$\Pi$ in a point $S$, hence the space $\langle T,Q^-(3,q) \rangle$ meets $Q^+(7,q)$ 
in a cone $S Q^-(3,q)$. Now the base of this cone is an ovoid of the base of the cone 
$S^\perp \cap Q^+(7,q) = S Q^+(5,q)$, so $\Pi$ contains exactly one line of the cone $SQ^-(3,q)$,
and hence $Q^-(3,q)$ meets $\Pi$, a contradiction. So, the points of $Q^-(3,q)$ are a maximal clique 
of $\cG_3$
\end{proof}

Before dealing with cliques of mixed type, we give a final lemma proving that small partial ovoids
determine either an elliptic quadric (hence a maximal partial ovoid of size $5$) or a complete
ovoid in a unique way. 

\begin{lemma}\label{le:4pts}
Let $q=2$. A partial ovoid $\cP$ of $Q^+(7,q)$ of size $4$ can be extended in a unique way to an ovoid.
\end{lemma}
\begin{proof}
Let $\cP = \{P_1, P_2, P_3, P_4\}$. We show that these four points determine uniquely 5 
points extending the partial ovoid to an ovoid. Define the following sets for $j \in \{2,3,4\}$ and $1 \leq i \leq j-1$
\[ 
A_{i}^j := P_j^\perp \cap P_i^\perp \cap Q^+(7,q) = \langle P_i,P_j\rangle^\perp \cap Q^+(7,q)\,.
\]
Hence, the set $A_{i}^j$ is the set of points collinear with both $P_i$ and $P_j$. 

If we make a choice for a point $P_j$, the next point $P_{j+1}$ cannot be collinear with $P_j$. Hence, from the list of possible 
points to extend the partial ovoid to an ovoid, we always have to remove the points collinear with $P_j$. Clearly, we will have
 already removed points collinear with $P_i$, $1 \leq i \leq j-1$. 

Now consider $P_1$. The points collinear with $P_1$ (including $P_1$) are the 
points of $P_1^\perp \cap Q^+(7,q) = P_1 Q^+(5,q)$. Hence, there are $135 - 71 = 64$ possible choices for $P_2$.
Now consider $P_2$. We have to remove the points of $P_2 Q^+(5,q)$. The line $l = \langle P_1,P_2 \rangle$ is a projective line of hyperbolic type, so $l^\perp  \cap Q^+(7,q) = Q^+(5,q)$, hence $|A_1^2| = 35$, and these points have been removed already 
after choosing $P_1$. So in total, we have to remove $71-35=36$ points. This leaves $64 - 36 = 28$ choices for $P_3$.
Now consider $P_3$. The plane $\langle P_1,P_2,P_3 \rangle$ is of parabolic type, since none of the three points 
$P_1, P_2, P_3$ are collinear. Therefore, $\pi = \langle P_1,P_2,P_3 \rangle^\perp$ is of parabolic type, i.e. 
$\pi \cap Q^+(7,q) = Q(4,q) = A_1^3 \cap A_2^3$. The number of points already removed after choosing 
$P_1$ and $P_2$ equals $|A_1^3 \cup A_2^3| = |A_1^3| + |A_2^3| - |A_1^3\cap A_2^3| = 35+35-15 = 55$.
So in total, we have to remove $71-55=16$ points. This leaves $28 - 16 = 12$ choices for $P_4$. 
The solid $\pi_3 := \langle P_1,P_2,P_3,P_4 \rangle$ is of elliptic type, so $\pi_3  \cap Q^+(7,q) = Q^-(3,q) = 
A_1^4 \cap A_2^4 \cap A_3^4$. The number of points already removed after choosing 
$P_1$, $P_2$, and $P_3$ equals $|A_1^4 \cup A_2^4 \cup A_3^4| = |A_1^4| + |A_2^4| + |A_3^4| - |A_1^4\cap A_2^4| - 
|A_1^4\cap A_3^4| - |A_2^4 \cap A_3^4| + |A_1^4 \cap A_2^4 \cap A_3^4| = 35+35+35-15-15-15+5 = 65$.
So in total, we have to remove $71-65=6$ points. This leaves a set $\cR$ of $12 - 6 = 6$ remaining points. 
The solid $\pi_3$ intersects $Q^+(7,q)$ in an elliptic quadric $Q^-(3,q)$. Its unique point different from $P_1,\ldots,P_4$ 
belongs to $\cR$. If we add this unique point to the partial ovoid, then it becomes a maximal partial ovoid by Lemma~\ref{le:elliptic1}.
Hence, if we want to extend $\cS$ to an ovoid, we can remove this unique point from $\cR$. Then $\cR$ is a set of $5$ points. 

Now assume that two points $S_1, S_2 \in \cR$ are collinear. The line $m = \langle S_1,S_2 \rangle$ is contained in $Q^+(7,q)$. 
But this line cannot be contained in $P_i^\perp$ for $i \in \{1,2,3,4\}$. So, necessarily, $P_i^\perp$ meets $m$ in a point not in $\cR$,
$m \setminus \cR = \{R\}$. This means that $P_i \in r^\perp$, $i \in \{1,2,3,4\}$. Now $r^\perp \cap Q^+(7,q) = RQ^+(5,q)$. The four points $P_i$, $i \in \{1,2,3,4\}$
are necessarily projected from $R$ onto a partial ovoid $\cR'$ of $Q^+(5,q)$. Hence, they span a solid $\pi_3$ meeting this $Q^+(5,q)$ in an elliptic quadric $Q^-(3,q)$. 
The line $m$ meets $Q^+(5,q)$ in a point, w.l.o.g. we may assume this is $S_1$. Then $S_1^\perp \cap Q^-(3,q)$ is necessarily a conic. Since $Q^-(3,q)$ contains
$5$ points in total, of which only one is not in $\cR'$, $S_1^\perp$ must contain at least one point of $\cR'$. Hence, there exists at least one plane of $Q^+(7,q)$ 
through the line $m$ containing one of the points $P_i$, a contradiction since none of these points is collinear with $S_1$ and $S_2$. We conclude
that no two points of $\cR$ are collinear. So, the set $\cR \cup \{P_1,P_2,P_3,P_4\}$ is an ovoid of $Q^+(7,q)$. 
\end{proof}

The strategy to classify maximal cliques in $\cG_3$ for $q=2$ will be to characterise very 
small cliques of mixed type in a geometric way. This will restrict the possibilities to extend small cliques to 
larger cliques. 

\begin{lemma}\label{triangle}
If $\cC = \{P,Q,R\}$ is a clique of mixed type in $\cG_3$, then there are exactly two $\sim_1$-adjacencies and one $\sim_2$-adjacency, and $\cC$ spans
a plane meeting $\Pi$ in a point $S$.
\end{lemma}
\begin{proof}
Let $P,Q \in \cC$ with $P \sim_1 Q$, then the line $\langle P,Q \rangle$ is a secant to $Q^+(7,q)$. We may assume w.l.o.g that $P \sim_2 R$. 
If $Q \sim_2 R$ as well, then the lines $\langle P,R \rangle$, respectively $\langle Q,R \rangle$ are contained in $Q^+(7,q)$ and meet $\Pi$ 
in the respective points $S$ and $T$. So, the line $\langle S,T \rangle$ is contained in $\Pi$ but also in the plane $\langle P,Q,R \rangle$, hence $\langle S,T \rangle$
also meets the line $\langle P,Q \rangle$ in a point of $Q^+(7,q)$, different from $P$ and $Q$, a contradiction. Hence $Q \sim_1 R$. 
But now it is not possible that the plane $\langle P,Q,R \rangle$ meets $\Pi$ in a line $m$, since then the line $\langle Q,R \rangle$ would also
meet $m$ in a point not in $Q^+(7,q)$, a contradiction. 
\end{proof}

Now we will restrict to $q=2$.

\begin{lemma}\label{le:singulars}
Assume that $\cC$ is a $\sim_2$-clique in $\cG_3$. Let 
$\alpha := \langle \cC \rangle$. Then $\alpha$ is contained in $Q^+(7,q)$,
and hence $\alpha$ is a plane or a solid, meeting $S$ in a line or in a plane respectively. No $\sim_2$-clique can span a space of dimension larger than $3$. 
\end{lemma}
\begin{proof}
Let $P,Q,R \in \cC$. Since $q=2$, $\langle P,Q,R \rangle$ is a plane. Each of 
the lines $\langle P,Q \rangle$, $\langle P,R \rangle$, and $\langle Q,R \rangle$
meets $\Pi$ in a different point, hence $\langle P,Q,R \rangle$ meets $\Pi$ in a line $m$.
If there exists another point $S \in \cC \setminus \langle P,Q,R \rangle$, then each of the lines connecting $S$ with one of the points $P,Q$, or $R$ meets $\Pi$ 
in a point not on $m$. Hence $\langle P,Q,R,S \rangle$ is then a solid, contained 
in $Q^+(7,q)$, meeting $\Pi$ in a plane. A similar argument for a point 
$T \in \cC \setminus \langle P,Q,R,S \rangle$ gives rise to a $4$-space contained
in $Q^+(7,q)$, a contradiction. 
\end{proof}

\begin{lemma}\label{le:ext1}
Let $\cC = \{P,Q,R\}$ be a clique of mixed type in $\cG_3$. Assume that
$T \not \in \langle P,Q,R\rangle$ extends $\cC$ to a larger clique
such that $T$ is $\sim_2$-adjacent with at least one of the vertices of $\cC$. 
Then $\gamma = \langle P,Q,R,T\rangle$
is a $3$-dimensional space and $\gamma \cap Q^+(7,q) = l Q^+(1,q)$, with
the line $l$ contained in $\Pi$, and there is only one way to extend 
$\{P,Q,R,T\}$ to a maximal clique of size $8$, which is then the set of 
vertices of $\cG_3$ in $l Q^+(1,q)$
\end{lemma}
\begin{proof}
By Lemma~\ref{triangle} we may assume that $P \sim_2 Q$ and $P \sim_1 R \sim_1 Q$, 
and the plane $\alpha := \langle P,Q,R \rangle$ meets $\Pi$ in a point $S$. Hence 
$\alpha \cap Q^+(7,q) = S Q^+(1,q)$. The line $\langle R,S \rangle$ contains one more
point $R'$ not in $\Pi$, for which clearly $R \sim_2 R'$ and $P \sim_1 R' \sim_1 Q$.

We have assumed that $T$ extends $\cC$ to a larger clique. Assume first that $T \sim_2 R$. 
Then the line $\langle T,R \rangle$ meets $\Pi$ in a point $S' \neq S$. 
By Lemma~\ref{triangle}, it is not possible that $T \sim_1 R'$. 
Hence the line $\langle T,R' \rangle$ is contained in $Q^+(7,q)$ and in the plane 
$\langle T,R,R' \rangle$, which meets $\Pi$ in the line $\langle S,S' \rangle$. Hence the line 
$\langle T,R' \rangle$ also meets $\Pi$ in a point of the line $\langle S,S' \rangle$, so we conclude
that $T \sim_2 R'$. 

It is not possible that $P \sim_2 T \sim_2 R$, since $P \sim_1 R$ and using Lemma~\ref{triangle}. 
Hence, the assumption that $T$ is $\sim_2$-adjacent with at least one of the points of $\cC$
implies that $T^\perp$ contains exactly one of the lines of 
$\alpha \cap Q^+(7,q) = S Q^+(1,q)$. W.l.o.g we may assume now 
that $T \sim_2 P$, $T \sim_2 Q$ and $T \sim_1 R$. Now define $\beta := \langle T,P,Q\rangle$. 
The clique $\{P,Q,T\}$ is a $\sim_2$-clique and by Lemma~\ref{le:singulars},
the plane $\beta$ is contained in $Q^+(7,q)$ and meets $\Pi$ in a line $l$, containing $S$. 
But then $\gamma := \langle \alpha,\beta \rangle$ is a solid containing a plane $\beta$
of $Q^+(7,q)$ and a point $R \not \in \beta$. Hence $\gamma \cap Q^+(7,q) = l Q^+(1,q)$, 

Consider now a vertex $U \not \in \gamma$ and assume that $U$ extends $\{P,Q,R,T\}$ 
to a larger clique. If $U^\perp$ meets $l$ in one 
point, then $U$ is collinear in $Q^+(7,q)$ with at least one of the points $P$, $Q$ or $T$. However,
the connecting line cannot meet $\Pi$, so there can be no adjacency $\sim_2$ or $\sim_1$ with 
at least one of the points $P$, $Q$, or $T$. Hence, if $U$ extends $\{P,Q,R,T\}$ to a larger clique,
then $l \subseteq U^\perp$. Hence, the $4$-space $\delta := \langle \{P,Q,R,T,U\} \rangle$ 
contains the planes $\beta$, $\beta'$ and $\beta'' := \langle U,l \rangle$, so
$\delta \cap Q^+(7,q) = l Q$, a cone with vertex $l$ and base $Q$, which is either a 
conic or a degenerate quadric in a plane. Now
$U \sim R$ and $P \sim U \sim Q$. But $l^\perp \cap Q^+(7,q) = l Q^+(3,q)$, and since 
$q=2$, the points of either $\beta$ or $\beta'$ are collinear with all the points of $\beta''$, 
including $U$. But the line connecting such a point
not on $l$ with $U$ will not meet $\Pi$. Hence, there can be no adjacency between $U$ and 
either the vertices in $\beta$ or in $\beta'$. 
Hence, the only possibility to extend further the clique $\{P,Q,R,T\}$ is by adding vertices 
in the planes $\beta$ and $\beta'$, and it is easy to see
that the vertices in both planes indeed make a clique of size $8$, the largest possible 
size. Finally, note that there would be no difference if the point $R'$ was used first to 
extend the initial clique $\{P,Q,R\}$.
\end{proof}

\begin{lemma}\label{le:ext2}
Let $\cC = \{P,Q,R\}$ be a clique of mixed type in $\cG_3$. Assume that
$T \not \in \langle P,Q,R\rangle$ extends $\cC$ to a larger clique such 
that $T$ is $\sim_1$-adjacent to all the vertices of $\cC$. Then 
$\gamma = \langle P,Q,R,T\rangle$
is a $3$-dimensional space and $\gamma \cap Q^+(7,q) = SQ(2,q)$, with
$S \in \Pi$, $Q(2,q)$ a conic not meeting $\Pi$ and there is only one way to extend 
$\{P,Q,R,T\}$ to a maximal clique of size $8$, which is then the set of 
vertices in a cone  $S Q^-(3,q)$ that meets $\Pi$ in a line.
\end{lemma}

\begin{proof}
By Lemma~\ref{triangle} we may assume that $P \sim_2 Q$ and $P \sim_1 R \sim_1 Q$, 
and the plane $\alpha := \langle P,Q,R \rangle$ meets $\Pi$ in a point $S$. Hence 
$\alpha \cap Q^+(7,q) = S Q^+(1,q)$. Since $P \sim_1 T \sim_1 Q$, $T^\perp$ meets
the line $\langle P,Q \rangle$ in the point $S$, and since $T \sim_1 R$, it is now clear
that $\gamma = \langle P,Q,R,T\rangle$ meets $Q^+(7,q)$ in a cone $S Q(2,q)$. W.l.o.g we may 
even assume that the base of this cone is the conic consisting of the points $\{P,R,T\}$.

Now assume that $U \not \in \gamma$ extends $\{P,Q,R,T\}$ to a larger clique. If $U$
is $\sim_2$-adjacent with one of the vertices of $\{P,Q,R,T\}$ then we can apply Lemma~\ref{le:ext1}. 
Assume e.g. that $U \sim_2 R$, then apply Lemma~\ref{le:ext1} on the mixed clique $\{P,Q,R\}$
and the vertex $U$ to conclude that $\gamma := \langle U,P,Q,R \rangle$ meets $Q^+(7,q)$ in a cone 
$lQ^+(1,q)$. However, this cone cannot contain the vertex $T$, while $T$ extends the clique 
$\{U,P,Q,R\}$ to a larger clique, a contradiction by Lemma~\ref{le:ext1}. This argument can be repeated
for any $\sim_2$-adjacency between $U$ and a vertex of $\{P,Q,R,T\}$. Hence, we must conclude
that $U$ is $\sim_1$-adjacent with all vertices of $\{P,Q,R,T\}$. But then the points
$\{P,R,T,U\}$ span a $3$-dimensional space meeting $Q^+(7,q)$ in a set of mutually non-collinear
points, hence $\langle P,R,T,U \rangle \cap Q^+(7,q) = Q^-(3,q)$. As we concluded that $U$
must be $\sim_1$ adjacent to $P$ and $Q$, necessarily $U^\perp$ meets the line $\langle P,Q \rangle$
in the point $S$. Hence $\langle P,Q,R,T,U \rangle \cap Q^+(7,q) = S Q^-(3,q)$. Now the base of
this cone is $Q^-(3,q)$, which is an ovoid of the base of the cone 
$S^\perp \cap Q^+(7,q) = S Q^+(5,q)$, so 
$\Pi$ contains exactly one line of the cone $SQ^-(3,q)$. Now we assume that $V$ extends the clique
$\{P,Q,R,T,U\}$ further. We can apply Lemma~\ref{le:ext1} again to conclude that a $\sim_2$-adjacency
between $V$ and any vertex of $\{P,Q,R,T,U\}$ is not possible. But now the vertices $\{P,R,T,U,V\}$
are a $\sim_1$-clique of size 5. Either this is maximal, or can be extended only with 
other $\sim_1$-adjacencies to a larger clique by Lemma~\ref{le:max5}, 
a contradiction now with the fact that the clique  $\{P,Q,R,T,U,V\}$ contains the adjacency 
$P \sim_2 Q$. Hence, the only way to extend the clique
$\{P,Q,R,T,U\}$ is to add the remaining vertices on the lines of the cone $S Q^-(3,q)$, which 
extends the clique to a clique of maximal size. 
\end{proof}

\begin{theorem}\label{main1}
Table~\ref{cliqueTab} lists all maximal cliques of the graph $\cG_3$. 
\end{theorem}
\begin{proof}
First, consider a set $\cC$ of $4$ points in general position and assume that $\cC$ is a $\sim_1$ 
clique. Then $\langle \cC\rangle \cap Q^+(7,2) = Q^-(3,q)$. If $\langle \cC \rangle$ meets $\Pi$ 
in a point, then this point is unique, and by Lemma~\ref{le:4pts} we can extend $\cC$ uniquely 
to an ovoid of $Q^+(7,2)$, which determines a maximal clique of size $8$. If $\langle \cC \rangle$
does not meet $\Pi$, then either adding the unique point of $Q^-(3,q) \setminus \cC$ gives rise to a
maximal clique of size $5$, by Lemma~\ref{le:max5}, or, by Lemma~\ref{le:4pts}, there is a unique
extension of $\cC$ to an ovoid, which again determines a clique of size $8$. 

By Lemma~\ref{le:singulars}, any $\sim_2$-clique $\cC$ must be contained in a subspace of the quadric. 
Hence, any $\sim_2$-clique can be extended to a $\sim_2$-clique of size $8$ consisting of all the
vertices contained in $\langle \cC \rangle$. 

Finally, assume that $\cC$ is a clique of mixed type. Then it is sufficient to consider a subset 
$\{P,Q,R\} \subset \cC$ which is a clique of mixed type and to use Lemmas~\ref{le:ext1}~and~\ref{le:ext2} 
to conclude that a maximal clique of mixed type must be one of the two examples listed in
Table~\ref{cliqueTab}.
\end{proof}

\subsection{The isomorphism issue}

We are now ready to state the non-isomorphism between $NO^+(2n+2,2)$ and $\cG_n$, while $n\geq3$. Next theorem collects all results about maximal cliques of $\cG_3$, in order to distinguish $NO^+(8,2)$ and $\cG_3$.


\begin{theorem}\label{main2}
    The graphs $\cG_n$ and $NO^+(2n+2,2)$ are not isomorphic when $n\geq3$.
\end{theorem}

\begin{proof}
Firstly, while $n=3$, $NO^+(8,2)$ contains only maximal 8-cliques, while $\cG_3$ contains a family of maximal 5-cliques, and they are not isomorphic. Now, let $n>3$ and consider a 7-space $\Sigma$ of $\PG(2n+1,2)$ such that $\Sigma\cap Q^+(2n+1,2)=Q^+(7,2)$, which meets the fixed generator $\Pi$ in a solid. Then the induced subgraphs by $NO^+(2n+2,2)$ and $\cG_n$ on $\Sigma$ are $NO^+(8,2)$ and $\cG_3$, in which we have maximal cliques of different sizes, by Theorem \ref{main1}.
\end{proof}

\begin{remark}
Note that for asymptotically large values of $n$ the graph $NO^+(2n+2,2)$ contains a subgraph $\cG_3$ and the graph $\cG_n$ contains a subgraph $NO^+(8,2)$. More generally, let $H$ be a graph with $s$ vertices and $r$ edges. The number of induced subgraphs $H$ in a random graph with $v$ vertices and with edge-probability (i.e. ratio between the number of edges and the number of unordered pair of vertices) equal to $p$ is
\begin{equation}\label{K-S}
(1+o(1))p^r(1-p)^{\binom{s}{2}-r}\frac{v^s}{|Aut(H)|}
\end{equation}
see \cite[Section 4.4]{KS}. Now we compute $|Aut(NO^+(8,2)|=|PGO^+(8,2)|=348364800$, $|Aut(\cG_3)|=1290240$. Both $NO^+(8,2)$ and $\cG_3$ have $s=120$ vertices and $r=3360$ edges, while both $NO^+(2n+2,2)$ and $\cG_n$ have $s=2^{2n+1}-2^n$ vertices and edge-probability $p=\frac{2^{2n}-1}{2^{2n+1}-2^n-1}$. For small $n$, Equation \eqref{K-S} gives asymptotically 0 subgraphs. On the other hand, each hyperbolic section $PG(8,2)$ of $PG(2n+2,2)$ gives rise to a $NO^+(8,2)$ in $NO^+(2n+2,2)$, while there is an induced $\cG_3$ in $\cG_n$ for any hyperbolic section cutting the fixed generator in a 3-space. Hence Theorem \ref{main2} holds. In fact, repeated inductive arguments show that $NO^+(2n+2,2)$ does not contain an induced $\cG_{n-1}$, and $\cG_n$ does not contain an induced $NO^+(2n,2)$, where the induction basis is still Theorem \ref{main1}.
\end{remark}

\begin{remark}
  According to the statement of \cite[Theorem 1]{BIK}, the graph $\cG_n$ for any $q$ has automorphism group $q^{\frac{n(n+1)}{2}}.PSL(n+1,q)$. Thus, the non-isomorphism would belong to the group size. The proof provided in Theorem \ref{main2} provides pure combinatorial and self-contained arguments that show the novelty of the results presented.
\end{remark}
\begin{remark}
    An alternative combinatorial proof for the non-isomorphism would take into account the induced subgraphs $\Gamma_{2}(v)$ on vertices at distance 2 from a fixed vertex $v\in\Gamma$. Then $NO^+(8,2)_2(v)$ is known to be distance-regular of diameter 3 while ${\mathcal G_{3}}_{,2}(v)$ has diameter 2 from the construction in Theorem \ref{mainth}. We presented the geometric proof in order to give a finite geometric interpretation of both strongly regular graphs. 
\end{remark}

\section{Conclusion}
In this paper, we provided the proof of \cite[Theorem 1]{BIK}, giving a construction of strongly regular graphs in hyperbolic quadrics $Q^+(2n+1,q)$. We have seen that, while $q=2$, such a family gives examples of graphs that are cospectral to $NO^+(2n+2,2)$, but non-isomorphic whenever $n\geq3$. In \cite{Btab}, there are some examples of graphs cospectral to $NO^{+}(2n+2,2)$. One can ask what are the isomorphism class and the switching class of $\cG_n$ beneath its cospectral graphs, and further research line would eventually study the isomorphisms and the geometric characterisation.





\smallskip
{\footnotesize
\noindent\textit{Acknowledgments.}
This work was supported by the Italian National Group for Algebraic and Geometric Structures and their Applications (GNSAGA-- INdAM), partially by the Scientific Research Network {\em Graphs, Association schemes and Geometries: structures, algorithms and computation} (W003324N) of the Research Foundation -- Flanders (Belgium), FWO, and also partially by VUB-OZR project "Discrete Structures, and their applications in Data Science” (OZR3637). The authors also wish to thank Ferdinand Ihringer for pointing out Remark \ref{ihr}.}

\end{document}